\documentclass[12pt]{article}
\usepackage{amssymb,amsmath,latexsym}
\usepackage{amsthm}
\newtheorem{theorem}{Theorem}

\begin{document}
\author{Alexander E Patkowski}
\title{Another Smallest Part Function related to Andrews' $spt$ function}
\date{\vspace{-5ex}}
\maketitle
\abstract{In this note, we offer some relations and congruences for two interesting $spt$-type functions, which together form a relation to Andrews' $spt$-function.}

\section{Introduction and Main Results}
In [3], we find the identity of Andrews' 
\begin{equation}\sum_{n\ge1}\frac{q^n}{(1-q^n)(q^n)_{\infty}}=\sum_{n\ge1}np(n)q^n+\frac{1}{(q)_{\infty}}\sum_{n\ge1}\frac{(-1)^nq^{n(3n+1)/2}(1+q^n)}{(1-q^n)^2}.\end{equation}
Here $p(n)$ is the number of partitions of $n,$ and the last series on the right side of (1) generates $N_2(n)=\sum_{m\in\mathbb{Z}}m^2N(m,n),$ $N(m,n)$ being the number of partitions of $n$ with rank $m$ [1]. The largest part minus the number of parts is defined to be the rank. The function $spt(n)$ counts the number of smallest parts among integer partitions of $n.$ Lastly, we used the familiar notation [7] $(a)_n=(a;q)_n=(1-a)(1-aq)\cdots(1-aq^{n-1}).$
\par In this note we find a $spt$-type function that is related to the generating function in (1) and falls into the same class of $spt$-type functions like the one offered in [9]. However, this note differs from [9] in that we will find the ``crank companion" to create a ``full" spt function related to Andrews' $spt$-function. Here we are also appealing to relations to $spt(n)$ modulo $2,$ whereas in [9] we concentrated on relations to $spt(n)$ modulo $3.$ Lastly, the partitions involved in this study are different, and deserve a separate study. Let $M_2(n)=\sum_{m\in\mathbb{Z}}m^2M(m,n),$ where $M(m,n)$ is the number of partitions of $n$ with crank $m$ [2].
\begin{theorem} We have,
\begin{equation}\sum_{n\ge1}\frac{q^n(q^{2n+1};q^2)_{\infty}}{(1-q^n)^2(q^{n+1})_{\infty}}=\frac{1}{(q^2;q^2)_{\infty}}\sum_{n\ge1}\frac{nq^n}{1-q^n}-\frac{1}{2}\sum_{n\ge1}N_2(n)q^{2n},\end{equation}

\begin{equation}\sum_{n\ge1}\frac{q^{n(n+1)/2}(q^{2n+1};q^2)_{\infty}}{(1-q^n)^2(q^{n+1})_{\infty}}=\frac{1}{(q^2;q^2)_{\infty}}\sum_{n\ge1}\frac{nq^n}{1-q^n}-\frac{1}{2}\sum_{n\ge1}M_2(n)q^{2n}.\end{equation}
\end{theorem}

For our next Theorem, which is a number-theoretic interpretation of Theorem 1, we will use the following definitions. We define a triangular 
partition to be of the form $\delta_l=(l-1, l-2, \cdots, 1),$ $l\in\mathbb{N}.$ Define the smallest part of a partition $\pi$ to be $s(\pi),$ and the largest part to be $l(\pi).$ We will also consider the partition pair $\sigma=(\pi,\delta_i),$ where we set $i=s(\pi).$ Since $i=s(\pi)$ in $\sigma$ we have that 
$s(\pi)-l(\delta_{i=s(\pi)})=s(\pi)-(s(\pi)-1)=1.$ If we include $\delta_i$ in a partition, we are increasing its size by $\binom{i}{2}$ and including the component $q^{1+2+\cdots +i-1}$ in its generating function. This has the property that all parts from $1$ to $i-1$ appear exactly once and are less than $i.$

\begin{theorem} Let $spt_o^{+}(n)$ count the number of smallest parts among the integer partitions $\pi$ of $n$ where odd parts greater than $2s(\pi)$ do not occur. Let $spt_o^{-}(n)$ count the number of smallest parts among the integer partitions $\sigma=(\pi,\delta_{s(\pi)})$ of $n$ where $\pi$ is a partition where odd parts greater than $2s(\pi)$ do not occur in $\pi.$  Define $spt_o(n):=spt_o^{+}(n)-spt_o^{-}(n).$ Then, $spt_o(2n)=spt(n).$
\end{theorem}

With the above definitions, we can write the generating function. We have

$$\sum_{n\ge1}spt_o(n)q^n=\sum_{n\ge1}\left(q^n+2q^{2n}+3q^{3n}+\cdots\right)\frac{(q^{2n+1};q^2)_{\infty}}{(q^{n+1})_{\infty}}\left(1-q^{1+2+\cdots n-1}\right).$$

\section{Proof of Theorem 1 and Theorem 2}
The proofs require the methods used in [6, 9] and a few more observations. A pair of sequences $(\alpha_n,\beta_n)$ is known to be a Bailey pair with respect to $a$ if 
		\begin{equation}\beta_n=\sum_{r\ge0}^{n} \frac{\alpha_r}{(aq;q)_{n+r} (q;q)_{n-r}}.\end{equation}
The next result is Bailey's lemma [4]. \newline
{\bf Bailey's Lemma} \it If $(\alpha_n,\beta_n)$ form a Bailey pair with respect to $a$ then
\begin{equation}\sum_{n\ge0}^{\infty}(\rho_1)_n(\rho_2)_n(aq/\rho_1\rho_2)^n\beta_n=\frac{(aq/\rho_1)_{\infty}(aq/\rho_2)_{\infty}}{(aq)_{\infty}(aq/\rho_1\rho_2)_{\infty}}\sum_{n\ge0}^{\infty}\frac{(\rho_1)_n(\rho_2)_n(aq/\rho_1\rho_2)^n\alpha_n}{(aq/\rho_1)_n(aq/\rho_2)_n}.\end{equation} \rm
The following are known Bailey pairs $(\alpha_n,\beta_n)$ relative to $a=1$ [10, C(1)]:
\begin {equation} \alpha_{2n+1}=0,\end{equation}
\begin {equation} \alpha_{2n}=(-1)^nq^{n(3n-1)}(1+q^{2n}),\end{equation}
\begin{equation} \beta_n=\frac{1}{(q)_n(q;q^2)_{n}},\end{equation}
and the pair [10, C(5)] $(\alpha_n,\beta_n),$ 
\begin {equation} \alpha_{2n+1}=0,\end{equation}
\begin {equation} \alpha_{2n}=(-1)^nq^{n(n-1)}(1+q^{2n}),\end{equation}
\begin{equation} \beta_n=\frac{q^{n(n-1)/2}}{(q)_n(q;q^2)_{n}}.\end{equation}
In both pairs $\alpha_0=1.$
\rm
Differentiating Bailey's lemma (putting $a=1$) with respect to both variables $\rho_1$ and $\rho_2$ and setting each variable equal to $1$ each time gives us [9],
\begin{equation}\sum_{n\ge1}(q)_{n-1}^2\beta_nq^n=\alpha_0\sum_{n\ge1}\frac{nq^n}{1-q^n}+\sum_{n\ge1}\frac{\alpha_nq^n}{(1-q^n)^2}.\end{equation}
The identity (2) follows from inserting the Bailey pair (6)-(8) into equation (9) and then multiplying through by $(q^2;q^2)_{\infty}^{-1}.$ Identity (3) follows from inserting the Bailey pair (9)-(10) into equation (9) and then multiplying through by $(q^2;q^2)_{\infty}^{-1}.$ This gives us Theorem 1. 
\\*
\par To get Theorem 2, we subtract equation (3) from equation (2), and note that $spt(n)=\frac{1}{2}(M_2(n)-N_2(n)),$ after noting that (see [6])
$$2\sum_{n\ge1}np(n)q^n=\sum_{n\ge1}M_2(n)q^n=\frac{2}{(q)_{\infty}}\sum_{n\ge1}\frac{nq^n}{1-q^n}.$$
 The result follows from equating coefficients of $q^{2n}.$

\section{More Notes and Concluding Remarks} 

\par Naturally, it is of interest to investigate equations (2) and (3) individually. As we noted previously, the left side of equation (2) generates $spt_o^{+}(n),$ and the left side of equation (3) generates $spt_o^{-}(n).$

\begin{theorem}We have, $spt_{o}^{+}(2n)\equiv spt(n)\pmod{2}.$\end{theorem}
\begin{proof} After noting that $\sigma(2n)=3\sigma(n)-2\sigma(n/2),$ $\sigma(2n)\equiv\sigma(n)\pmod{2},$ and 
\begin{equation}\sum_{n\ge1}spt_{o}^{+}(n)q^n=\frac{1}{(q^2;q^2)_{\infty}}\sum_{n\ge1}\sigma(n)q^n-\frac{1}{2}\sum_{n\ge1}N_2(n)q^{2n},\end{equation}
we can take coefficients of $q^{2n}$ in (13) to get
\begin{equation}spt_{o}^{+}(2n)=\sum_{k}p(k)\sigma(2(n-k))-\frac{1}{2}N_2(n).\end{equation}
Hence, combining these notes, we compute
\begin{align} spt_{o}^{+}(2n)\\
&\equiv \sum_{k}p(k)\sigma(n-k)-\frac{1}{2}N_2(n)\pmod{2} \\
&\equiv np(n)-\frac{1}{2}N_2(n)\pmod{2}\\
&\equiv spt(n)\pmod{2}.
\end{align}
\end{proof}
\begin{theorem}We have, $spt_{o}^{-}(2n)\equiv 0\pmod{2}.$\end{theorem}
\begin{proof} The computations are similar to Theorem 3. Using equation (3) we compute
\begin{align} spt_{o}^{-}(2n)\\
&\equiv \sum_{k}p(k)\sigma(n-k)-\frac{1}{2}M_2(n)\pmod{2} \\
&\equiv np(n)-\frac{1}{2}M_2(n)\pmod{2}\\
&\equiv 0\pmod{2}.
\end{align}
In line (22) we used $2np(n)=M_2(n).$
\end{proof}

For an example of Theorem 3, consider a partition of $4:$ $(4),$ $(3,1),$ $(2,2),$ $(1,1,1,1).$ In a partition where odd parts greater than twice the smallest do not occur, we omit $(3,1).$ Hence $spt_o^{+}(4)=7,$ and $spt(2)=3$ (counting smallest of $(2)$ and $(1,1)$). Hence $2$ divides $7-3=4.$ \par For an example of Theorem 4, consider the partition pair $\sigma=(\pi^{*}, \delta_{s(\pi^{*})})$ of 
$6:$ $(3,2)\in\pi^{*},$ $(1)\in\delta_{s(\pi^{*})},$ and $(3)\in\pi^{*},$ $(2,1)\in\delta_{s(\pi^{*})}.$ Hence $spt_{o}^{-}(6)$ is equal to $2$ plus the appearances of the smallest parts in $\pi^*$ of those partition pairs $\sigma=(\pi^{*}, \delta_{s(\pi^{*})})$ which have the empty partition $\emptyset\in\delta_i.$ That is, the pairs $(3,1,1,1)\in\pi^{*},$ $\emptyset\in\delta_{s(\pi^{*})};$ $(4,1,1)\in\pi^{*},$ $\emptyset\in\delta_{s(\pi^{*})};$ $(2,1,1,1,1)\in\pi^{*},$ $\emptyset\in\delta_{s(\pi^{*})};$ $(1,1,1,1,1,1)\in\pi^{*},$ $\emptyset\in\delta_{s(\pi^{*})};$ and finally $(3,2,1)\in\pi^{*},$ $\emptyset\in\delta_{s(\pi^{*})}.$ This gives us $spt_{o}^{-}(6)=18=0\pmod{2}.$
\\*
Equating coefficients of $q^{2n+1}$ in Theorem 1 gives us a nice corollary.
\begin{theorem}
We have, $spt_o^{-}(2n+1)=spt_o^{+}(2n+1).$
\end{theorem}
\par Let $t_k(n)$ be the number of representations of $n$ as a sum of $k$ triangular numbers. We may use a classical result of Legendre that $\sigma(2n+1)=t_4(n),$ to see that $spt_o^{-}(2n+1)$ (and therefore also $spt_o^{+}(2n+1)$) is generated by
the product expansion
\begin{equation}q\frac{(q^4;q^4)_{\infty}^3}{(q^2;q^4)_{\infty}^5}.\end{equation}
\par To see examples of Theorem 5, consider first $n=1.$ We have that $spt_o^{-}(3)=spt_o^{+}(3)=5.$ This is because $(2,1)\in\pi^{*},$ $\emptyset\in\delta_{s(\pi^{*})};$ $(1,1,1)\in\pi^{*},$ $\emptyset\in\delta_{s(\pi^{*})};$ and $(2)\in\pi^{*},$ $(1)\in\delta_{s(\pi^{*})},$ for $spt_o^{-}(3).$ The $spt_o^{+}(3)$ is clearer.
\par Another example is $spt_o^{-}(5)=spt_o^{+}(5)=12.$ We only compute $spt_o^{-}(5)$ for the reader. We compute $(2,2)\in\pi^{*},$ $(1)\in\delta_{s(\pi^{*})};$ $(2,2,1)\in\pi^{*},$ $\emptyset\in\delta_{s(\pi^{*})};$ $(4,1)\in\pi^{*},$ $\emptyset\in\delta_{s(\pi^{*})};$ $(1,1,1,1,1)\in\pi^{*},$ $\emptyset\in\delta_{s(\pi^{*})};$ and finally $(2,1,1,1)\in\pi^{*},$ $\emptyset\in\delta_{s(\pi^{*})}.$

\par  It is interesting to note that since $spt(n)$ is even for almost all natural $n$ [5], we have that $spt_{o}^{+}(2n)$ is even for almost all natural $n$ in terms of arithmetic density. 
\par We can also easily obtain congruences for $spt_o(n)$ using the Ramanujan-type congruences in [3]:
\begin{equation} spt_{o}(2(5n+4))\equiv 0\pmod{5},\end{equation}
\begin{equation} spt_{o}(2(7n+5))\equiv 0\pmod{7},\end{equation}
\begin{equation} spt_{o}(2(13n+6))\equiv 0\pmod{13}.\end{equation}
It is important to make the observation that the two Bailey pairs (6)--(8) and (9)--(11) are key in obtaining the ``rank component" (2) and the``crank component" (3), respectively.

1390 Bumps River Rd. \\*
Centerville, MA
02632 \\*
USA \\*
E-mail: alexpatk@hotmail.com

\begin{thebibliography}{9}
\bibitem{ConcreteMath}

G. E. Andrews, \emph{The Theory of Partitions}, The Encyclopedia of Mathematics and its Applications, Vol. 2, Addison-Wesley, Reading (1976). 
\bibitem{ConcreteMath}
G. E. Andrews and F. Garvan, \emph{Dyson's crank of a partition,} Bull. Amer. Math. Soc. (N.S.) 18 (1988), no. 2, 167--171.
\bibitem{ConcreteMath}
G. E. Andrews, \emph{The number of smallest parts in the partitions of n,} J. Reine Angew. Math. 624 (2008), 133--142.
\bibitem{ConcreteMath}
 W. N. Bailey, \emph{Identities of the Rogers--Ramanujan type}, Proc. London Math. Soc. (2), 50 (1949), 1--10.
\bibitem{ConcreteMath}
A. Folsom and K. Ono, \emph{The spt-function of Andrews,} Proc. Natl. Acad. Sci. USA 105 (2008), 20152--20156.
\bibitem{ConcreteMath}
F. Garvan, \emph{Higher Order spt--Functions,} Adv. in Math. 228 (2011), 241--265.
\bibitem{ConcreteMath}
 G. Gasper,  M. Rahman, \emph{Basic hypergeometric series}, Cambridge Univ. Press, Cambridge, 1990.
\bibitem{ConcreteMath}
A. E. Patkowski, \emph{Divisors, partitions and some new q-series identities,} Colloq. Math. 117 (2009), 289--294.
\bibitem{ConcreteMath}
A. E. Patkowski, \emph{A Strange Partition Theorem Related to the Second Atkin-Garvan Moment,} Preprint.
\bibitem{ConcreteMath}
L. J. Slater, \emph{A new proof of Rogers' transformations of infinite series,} Proc. London Math. Soc. (2) 53 (1951), 460--475.
\end{thebibliography}
\end{document}